\documentclass[a4paper,11pt]{article}

\usepackage[margin=1in]{geometry}
\usepackage{setspace}
\usepackage{amsmath}
\usepackage{amssymb}
\usepackage{amsthm}
\usepackage{xcolor}
\usepackage{hyperref}

\newtheorem{theorem}{Theorem}[section]
\newtheorem{proposition}[theorem]{Proposition}
\newtheorem{corollary}[theorem]{Corollary}
\newtheorem{lemma}[theorem]{Lemma}

\hypersetup{
    hidelinks,
    pdftitle={Spectral Estimates for Compact Riemann Surfaces via K\"ahler Potentials},
    pdfauthor={Hanwen Liu}
}

\begin{document}

\title{\texorpdfstring{\textbf{Spectral Estimates for Compact Riemann Surfaces via K\"ahler Potentials}}{Spectral Estimates for Compact Riemann Surfaces via K\"ahler Potentials}}
\author{Hanwen Liu\thanks{ORCID: 0009-0007-2503-9576}\\
{\small Mathematics Institute, University of Warwick}\thanks{University of Warwick, Coventry, United Kingdom, CV4 7AL}\\
{\small \href{mailto:hanwen.liu@warwick.ac.uk}{\texttt{hanwen.liu@warwick.ac.uk}}}}
\date{}
\maketitle

\begin{abstract}
We establish quantitative comparisons between the Laplace--Beltrami spectra of K\"ahler metrics of equal area on a compact Riemann surface. An estimate in terms of the oscillation of a potential gives bounds for fractional powers of reciprocal eigenvalues and explicit intervals for eigenvalue ratios. Bounds involving the gradient and Laplacian of the potential give further comparisons. On the Riemann sphere, we obtain estimates for individual eigenvalues, reciprocal sums and counting functions. For a K\"ahler metric on the unit 2-sphere which is centered in the sense that the unit normal vector field integrates to zero, the Dirichlet energy of the potential yields quantitative improvements of Hersch's bounds for the first positive eigenvalue and the sum of the first three reciprocal eigenvalues.
\end{abstract}

\begin{center}
\textbf{Keywords:} Laplace--Beltrami eigenvalues; compact Riemann surfaces; K\"ahler potentials.

\textbf{Mathematics Subject Classification:} Primary 58J50; Secondary 35P15, 53C55.
\end{center}

\tableofcontents
\onehalfspacing
\raggedbottom

\section{Introduction and Background}

The conformal invariance of the Dirichlet integral is a basic tool in the spectral geometry of compact surfaces. In 1970, Hersch~\cite{Hersch} proved sharp area bounds for the first positive eigenvalue on the sphere and the sum of the first three reciprocal eigenvalues. Using holomorphic maps to the sphere, Yang and Yau~\cite{YangYau} bounded the first positive eigenvalue on a compact Riemann surface in terms of area and genus. Karpukhin, Nadirashvili, Penskoi and Polterovich~\cite{KNPP} established the optimal area-normalized upper bound for every positive eigenvalue index on the sphere.

The present work compares two metrics in a fixed conformal class. Any two K\"ahler forms of equal area on a compact connected Riemann surface differ by $dd^c\phi$. The potential is unique up to a constant, so its oscillation is determined by the two forms. We use the oscillation and derivatives of the potential to estimate spectral changes.

Spectral comparison through a fixed energy form and varying measures has substantial precedents. Kokarev~\cite[Section~4.2]{Kokarev} studies continuity in an integral distance, while Girouard, Karpukhin and Lagac\'e~\cite[Section~4]{GKL} establish continuity under suitable dual Sobolev convergence. Burenkov, Gol'dshtein and Ukhlov~\cite{BGU} use weighted-form comparisons for Dirichlet spectra of conformally parametrized planar domains. Our principal estimate gives an explicit comparison at every positive index using only the oscillation of a Poisson potential, without requiring pointwise closeness of the area densities.

The main result is the ordered comparison in Theorem~\ref{thm:joint}. Its proof combines conformal invariance, integration by parts and the min--max principle. Multiplication by a positive function of the potential compares the Rayleigh quotients. Theorem~\ref{thm:p} bounds fractional powers of reciprocal eigenvalues, and Corollary~\ref{cor:interval} gives intervals for eigenvalue ratios. Theorem~\ref{thm:gradient} gives a further comparison in terms of the gradient of the potential. A bound for its Laplacian controls the ratio of the area forms. Pullback gives the oscillation comparison for metrics induced on holomorphically immersed curves.

On the Riemann sphere $\mathbb{P}^1$, the Fubini--Study reference spectrum and the classical sphere inequalities give bounds for individual eigenvalues, reciprocal sums and counting functions. For a K\"ahler metric on the unit 2-sphere which is centered in the sense that the unit normal vector field integrates to zero, Theorems~\ref{thm:sphere-energy} and~\ref{thm:reciprocal-energy} improve Hersch's inequalities by terms involving the Dirichlet energy of the potential. Both improvements are strict for a nonconstant potential. The first proof uses auxiliary Poisson equations; the second combines a spherical harmonic estimate with a finite rank comparison for the operator whose eigenvalues are the reciprocal Laplace eigenvalues.

The energy refinements belong to the quantitative stability theory of Hersch's inequality developed by Karpukhin, Nahon, Polterovich and Stern~\cite[Theorem~1.2]{KNPS}. They control a squared negative Sobolev distance from a round measure after a conformal automorphism. The contribution here is the fixed-area formula and a remainder for Hersch's three-term reciprocal sum; the relation between the norms and the explicit coefficients is discussed in Subsection~\ref{subsec:energy}.

\section{Comparisons in a Conformal Class}

As usual, we denote by $\sup(\psi)$ the supremum and $\sup(\psi)+\sup(-\psi)$ the oscillation of a real-valued function $\psi$.

Fix once and for all a compact connected Riemann surface $X$. All smooth functions on $X$ introduced below are real-valued. We use the convention
$$
dd^c=\sqrt{-1}\partial\bar\partial.
$$
For a K\"ahler form $\omega$ on $X$, we introduce the following notation.

The eigenvalues considered in this article are those of the nonnegative Laplacian $-\Delta_\omega$:
$$
0=\lambda_0(\omega)<\lambda_1(\omega)\le\lambda_2(\omega)\le\cdots,
$$
where every positive eigenvalue is counted with multiplicity. Suppose that $\omega_1,\omega_2$ are K\"ahler forms on $X$ such that there exists a non-constant $\phi\in C^\infty(X)$ satisfying
\begin{equation}\label{eq:setting}
\omega_2=\omega_1+dd^c\phi.
\end{equation}

The following observation identifies equation~$(\ref{eq:setting})$ with the equality of total areas.

\begin{lemma}
The area forms $\omega_1,\omega_2$ have equal integrals. Conversely, any two K\"ahler forms $\omega_1,\omega_2$ of equal total area on $X$ satisfy equation~$(\ref{eq:setting})$ for some $\phi\in C^\infty(X)$, unique up to an additive real constant.    
\end{lemma}

\begin{proof}
The equality of areas follows from Stokes' theorem. Conversely, write $\omega_2-\omega_1=\rho\omega_1$. Since the $2$-form $\rho\omega_1$ has zero integral, by virtue of Hodge theory, we obtain a smooth solution of the Poisson equation $\Delta_{\omega_1}\phi=2\rho$, unique up to an additive real constant. The identity $2dd^c\phi=(\Delta_{\omega_1}\phi)\omega_1$ proves the assertion.
\end{proof}

We next express the conformal invariance of the Dirichlet energy. This allows the two spectral problems to be compared through their area forms.

Denote by $\langle-,-\rangle_i$ and $|\cdot|_i$ the inner product and norm on the tensor bundle of $X$ induced by $\omega_i$. For $u,v\in C^\infty(X)$, define the conformally invariant Dirichlet form by
$$
Q_i(u,v)=\int_X\langle du,dv\rangle_i\omega_i.
$$
For simplicity, we write $Q_i(u)=Q_i(u,u)$. Since the forms $\omega_1,\omega_2$ are conformal, we have that $Q_1=Q_2$, and we write $Q:=Q_1=Q_2$. 

Integration by parts gives
\begin{equation}\label{eq:ibp}
\int_X udd^c\phi=-\frac12Q(\phi,u).
\end{equation}

\subsection{Oscillation of potentials}

We first compare corresponding eigenvalues in a prescribed order, while interchanging the forms removes the ordering hypothesis. The following estimate is the basis of subsequent results.

\begin{theorem}\label{thm:joint}
For every $j\ge1$ such that $\lambda_j(\omega_2)\geq\lambda_j(\omega_1)$, it holds that
\begin{equation}\label{eq:ordered}
\frac{1}{\sqrt{\lambda_j(\omega_1)}}\left(\frac{1}{\sqrt{\lambda_j(\omega_1)}}-\frac{1}{\sqrt{\lambda_j(\omega_2)}}\right)\leq\frac{\sup(\phi)+\sup(-\phi)}{2}.
\end{equation}
\end{theorem}

\begin{proof}
\binoppenalty=10000\relpenalty=10000
Put $a:=1/\lambda_j(\omega_1)$ and $L:=\sup(\phi)+\sup(-\phi)$. If $a\leq L/2$, then the assertion is immediate. We therefore assume that $a>L/2$, and set $\varphi:=a+(\phi-\sup(\phi))/2$. Thus, we have that $0<a-L/2\leq\varphi\leq a$. By equation~$(\ref{eq:ibp})$, we obtain that
\begin{equation}\label{eq:weighted-energy}
Q(\varphi v)=\int_X\varphi^2|dv|_1^2\omega_1+\int_X\varphi v^2(\omega_1-\omega_2).
\end{equation}
Let $V$ be the linear span of the first $j+1$ elements of an $L^2$-orthonormal eigenbasis of $-\Delta_{\omega_1}$, ordered by eigenvalue. For every $\varphi v\in V$, we have that
\begin{equation}\label{eq:reference-subspace}
Q(\varphi v)\leq\frac1a\int_X\varphi^2v^2\omega_1.
\end{equation}
Combining equations~$(\ref{eq:weighted-energy})$ and~$(\ref{eq:reference-subspace})$, we obtain that
\begin{equation}\label{eq:weighted-rayleigh}
(a-L/2)^2Q(v)\leq\int_X\varphi(\varphi/a-1)v^2\omega_1+\int_X\varphi v^2\omega_2\leq a\int_Xv^2\omega_2.
\end{equation}
Since the space $\varphi^{-1}V$ has dimension $j+1$, the min--max principle applied to equation~$(\ref{eq:weighted-rayleigh})$ gives $\lambda_j(\omega_2)\leq a/(a-L/2)^2$. Rearranging this inequality proves equation~$(\ref{eq:ordered})$.
\end{proof}

The estimate above also controls fractional powers of reciprocal eigenvalues.

\begin{theorem}\label{thm:p}
For every real $p\geq1$ and every integer $j\geq1$, it holds that
\begin{equation}\label{eq:allp}
\frac{1}{2}\left|\frac{1}{\lambda_j(\omega_1)^{1/p}}-\frac{1}{\lambda_j(\omega_2)^{1/p}}\right|^p
\le\frac{\sup(\phi)+\sup(-\phi)}{p+2-|p-2|}.
\end{equation}
\end{theorem}

\begin{proof}
\binoppenalty=10000\relpenalty=10000
By interchanging the forms and replacing $\phi$ by $-\phi$ if necessary, we may assume that $a:=1/\lambda_j(\omega_1)\geq b:=1/\lambda_j(\omega_2)$. Put $r:=\sqrt{b/a}$ and $L:=\sup(\phi)+\sup(-\phi)$. By equation~$(\ref{eq:ordered})$, we have that $2a(1-r)\leq L$. For $1\leq p\leq2$, we obtain
\begin{equation}\label{eq:fractional-elementary}
(1-r^{2/p})^p\leq1-r^{2/p}\leq\frac2p(1-r).
\end{equation}
For $p\geq2$, we have that $r^{2/p}\geq r$, and hence we obtain $$(1-r^{2/p})^p\leq(1-r)^p\leq1-r.$$ Together with equation~$(\ref{eq:fractional-elementary})$, this gives
\begin{equation}\label{eq:fractional-difference}
|a^{1/p}-b^{1/p}|^p=a(1-r^{2/p})^p\leq\frac{2a(1-r)}{\min\{p,2\}}\leq\frac{L}{\min\{p,2\}}.
\end{equation}
Since $p+2-|p-2|=2\min\{p,2\}$, equation~$(\ref{eq:fractional-difference})$ proves $(\ref{eq:allp})$.
\end{proof}

For comparison with a fixed reference metric, it is useful to solve the ordered estimate for the eigenvalue ratio.

\begin{corollary}\label{cor:interval}
For $j\ge1$ and $L:=\sup(\phi)+\sup(-\phi)$, it holds that
\begin{equation}\label{eq:interval}
\max\{0,2-\lambda_j(\omega_2)L\}^2
\le4\frac{\lambda_j(\omega_2)}{\lambda_j(\omega_1)}
\le\left(1+\sqrt{1+2\lambda_j(\omega_2)L}\right)^2.
\end{equation}
\end{corollary}

\begin{proof}
\binoppenalty=10000\relpenalty=10000
Put $r:=\sqrt{\lambda_j(\omega_2)/\lambda_j(\omega_1)}$ and $s:=\lambda_j(\omega_2)L/2$. If $r\geq1$, then equation~$(\ref{eq:ordered})$ gives $r(r-1)\leq s$. If $r\leq1$, then equation~$(\ref{eq:ordered})$, with the forms interchanged, gives $1-r\leq s$. In either case, we obtain that
\begin{equation}\label{eq:ratio-root}
\max\{0,1-s\}\leq r\leq\frac{1+\sqrt{1+4s}}2.
\end{equation}
Squaring equation~$(\ref{eq:ratio-root})$ and multiplying by $4$ proves equation~$(\ref{eq:interval})$.
\end{proof}

The same comparison applies to metrics induced on holomorphically immersed curves. The right-hand sides can be bounded using the ambient potential alone.

\begin{theorem}
Let $M$ be a compact complex manifold, and let $f\colon X\rightarrow M$ be a holomorphic immersion. Suppose that $\Omega_1,\Omega_2$ are K\"ahler forms on $M$ such that there exists $\Phi\in C^\infty(M)$ satisfying 
\begin{equation}\label{eq:setting2}
\Omega_2=\Omega_1+dd^c\Phi.
\end{equation}
Then, for every real $p\geq1$ and every integer $j\geq1$ such that $\lambda_j(f^*\Omega_2)\geq\lambda_j(f^*\Omega_1)$, it holds that
$$
\frac{1}{\sqrt{\lambda_j(f^*\Omega_1)}}\left(\frac{1}{\sqrt{\lambda_j(f^*\Omega_1)}}-\frac{1}{\sqrt{\lambda_j(f^*\Omega_2)}}\right)\leq\frac{\sup(\Phi)+\sup(-\Phi)}{2},
$$
and that 
$$
\frac{1}{2}\left|\frac{1}{\lambda_j(f^*\Omega_1)^{1/p}}-\frac{1}{\lambda_j(f^*\Omega_2)^{1/p}}\right|^p
\le\frac{\sup(\Phi)+\sup(-\Phi)}{p+2-|p-2|}.
$$
\end{theorem}

\begin{proof}
Since the map $f$ is a holomorphic immersion, the pullback forms $f^*\Omega_1,f^*\Omega_2$ are K\"ahler. Pulling back equation~$(\ref{eq:setting2})$, we obtain that $f^*\Omega_2=f^*\Omega_1+dd^c(\Phi\circ f)$.
Moreover, we have that $\sup(\Phi\circ f)+\sup(-\Phi\circ f)\leq\sup(\Phi)+\sup(-\Phi)$. If $\Phi\circ f$ is constant, then the two pullback forms coincide and both assertions are immediate. Otherwise, the assertions follow from Theorems~\ref{thm:joint} and~\ref{thm:p}.
\end{proof}

\subsection{Derivative estimates}

We now consider the first derivatives of the potential. Integration by parts and the min--max principle give the following comparison of reciprocal eigenvalues.

\begin{theorem}\label{thm:gradient}
For every $j\geq1$, the inequality
\begin{equation}\label{eq:gradient-strong}
\left|\frac{1}{\lambda_j(\omega_1)}-\frac{1}{\lambda_j(\omega_2)}\right|\leq\min\left\{\frac{\sup|d\phi|_1}{\sqrt{\lambda_j(\omega_1)}},\frac{\sup|d\phi|_2}{\sqrt{\lambda_j(\omega_2)}}\right\}
\end{equation}
holds. Consequently, it holds that
\begin{equation}\label{eq:gradient-root}
\left|\frac{1}{\sqrt{\lambda_j(\omega_1)}}-\frac{1}{\sqrt{\lambda_j(\omega_2)}}\right|
\leq\frac{\sup|d\phi|_1\sup|d\phi|_2}{\sup|d\phi|_1+\sup|d\phi|_2}.
\end{equation}
\end{theorem}

\begin{proof}
\binoppenalty=10000\relpenalty=10000
Put $a:=\sup|d\phi|_1$ and $\lambda:=\lambda_j(\omega_1)$. For $i=1,2$, define
$$
E_i(u):=\int_Xu^2\omega_i.
$$
By equation~$(\ref{eq:ibp})$ and the Cauchy--Schwarz inequality, we obtain that
\begin{equation}\label{eq:gradient-mass-difference}
|E_2(u)-E_1(u)|=\left|\int_Xu\langle d\phi,du\rangle_1\omega_1\right|\leq a\sqrt{E_1(u)Q(u)}.
\end{equation}
For $u\neq0$, put $R_i(u):=Q(u)/E_i(u)$. Equation~$(\ref{eq:gradient-mass-difference})$ gives $$R_2(u)\geq \frac{R_1(u)}{1+a\sqrt{R_1(u)}}.$$ The function $s/(1+a\sqrt{s})$ is increasing for $s\geq0$. By the min--max principle, we therefore obtain that $\lambda_j(\omega_2)\geq\lambda/(1+a\sqrt{\lambda})$.

If $a\sqrt{\lambda}<1$, then equation~$(\ref{eq:gradient-mass-difference})$ on the span of the first $j+1$ eigenfunctions of $-\Delta_{\omega_1}$ gives $E_2(u)\geq(1-a\sqrt{\lambda})E_1(u)$. The min--max principle therefore gives $\lambda_j(\omega_2)\leq\lambda/(1-a\sqrt{\lambda})$. If $a\sqrt{\lambda}\geq1$, then the quantity $1/\lambda-a/\sqrt{\lambda}$ is nonpositive. In either case, we obtain that
\begin{equation}\label{eq:gradient-reciprocal-interval}
\frac1\lambda-\frac{a}{\sqrt\lambda}\leq\frac1{\lambda_j(\omega_2)}\leq\frac1\lambda+\frac{a}{\sqrt\lambda}.
\end{equation}
Interchanging the forms in equation~$(\ref{eq:gradient-reciprocal-interval})$ proves equation~$(\ref{eq:gradient-strong})$.

Finally, put $b:=\sup|d\phi|_2$. Also write $c:=1/\sqrt{\lambda_j(\omega_1)}$ and $d:=1/\sqrt{\lambda_j(\omega_2)}$. By equation~$(\ref{eq:gradient-strong})$ and the elementary bound for the minimum, we have that
\begin{equation}\label{eq:gradient-root-algebra}
|c-d|=\frac{|c^2-d^2|}{c+d}\leq\frac{\min\{ac,bd\}}{c+d}\leq\frac{ab}{a+b}.
\end{equation}
For the last inequality, if $ac\leq bd$, then we have that $ac(a+b)\leq ab(c+d)$; the other case is analogous. Thus, equation~$(\ref{eq:gradient-root-algebra})$ proves equation~$(\ref{eq:gradient-root})$.
\end{proof}

A bound for the Laplacian of the potential gives a pointwise comparison of the area forms. Applying the min--max principle to this comparison gives the following estimate for eigenvalue ratios.

\begin{proposition}
Suppose that $\sup(-\Delta_{\omega_1}\phi)<2$. Then, for every $j\geq1$, it holds that
\begin{equation}\label{eq:density-comparison}
\frac{\lambda_j(\omega_1)}{2+\sup(\Delta_{\omega_1}\phi)}\leq\frac{\lambda_j(\omega_2)}{2}\leq\frac{\lambda_j(\omega_1)}{2-\sup(-\Delta_{\omega_1}\phi)}.
\end{equation}
In particular, if $\|\Delta_{\omega_1}\phi\|_\infty\leq2\varepsilon<2$, then, for every $j\geq1$, it holds that
\begin{equation}\label{eq:second-order}
\frac1{1+\varepsilon}\leq\frac{\lambda_j(\omega_2)}{\lambda_j(\omega_1)}\leq\frac1{1-\varepsilon}.
\end{equation}
\end{proposition}

\begin{proof}
\binoppenalty=10000\relpenalty=10000
By equation~$(\ref{eq:setting})$, we have that $2\omega_2=(2+\Delta_{\omega_1}\phi)\omega_1$. Consequently, for every smooth function $u\in C^\infty(X)$, we obtain that
\begin{equation}\label{eq:density-mass}
\bigl(2-\sup(-\Delta_{\omega_1}\phi)\bigr)\int_Xu^2\omega_1\leq2\int_Xu^2\omega_2\leq\bigl(2+\sup(\Delta_{\omega_1}\phi)\bigr)\int_Xu^2\omega_1.
\end{equation}
By the conformal invariance of $Q$, the min--max principle applied to equation~$(\ref{eq:density-mass})$ immediately proves $(\ref{eq:density-comparison})$.

If $\|\Delta_{\omega_1}\phi\|_\infty\leq2\varepsilon<2$, then both suprema in equation~$(\ref{eq:density-comparison})$ are at most $2\varepsilon$. Substituting these bounds gives equation~$(\ref{eq:second-order})$.
\end{proof}

\section{Quantitative Bounds for the Sphere}

We now use the explicit spectrum of the round sphere as a reference. Let $\hat\omega$ be the Fubini--Study metric on $\mathbb{P}^1$ normalized by
$$
\int_{\mathbb{P}^1}\hat\omega=2\pi.
$$
For simplicity, we denote by $\Delta$ the Laplace--Beltrami operator of $(\mathbb{P}^1,\hat\omega)$. We also denote by $\langle-,-\rangle$ and $|\cdot|$ the inner product and norm on the tensor bundle of $\mathbb{S}^2$ induced by $\hat\omega$.

The Gaussian curvature of $(\mathbb{P}^1,\hat\omega)$ is $2$, and the spectrum of $-\Delta$ is
\begin{equation}\label{eq:round}
\lambda_j(\hat\omega)=2k(k+1)
\end{equation}
whenever $k^2\le j\le(k+1)^2-1$ and $k\ge1$.
Thus, the first positive Laplace–Beltrami eigenvalue of $\hat\omega$ is $4$ with multiplicity $3$, followed by $12$ with multiplicity $5$.

Take a smooth function $\phi\in C^\infty(\mathbb{S}^2)$ such that
$$
\omega:=\hat\omega+dd^c\phi>0, \qquad L:=\sup(\phi)+\sup(-\phi)>0.
$$
\subsection{Comparison with the round spectrum}

The next theorem combines Corollary~\ref{cor:interval} with Hersch's inequalities. The upper bound for the first positive eigenvalue and the lower bound for the reciprocal sum are classical; the opposite bounds depend on the oscillation of the potential.

\begin{theorem}\label{thm:line}
The Laplace–Beltrami spectrum of $\omega$ satisfies
\begin{equation}\label{eq:first}
4\geq\lambda_1(\omega)\geq\frac{16}{(1+\sqrt{1+8L})^2}
\end{equation}
and
\begin{equation}\label{eq:sum}
\frac14\le \frac 13\sum_{j=1}^{3}\frac{1}{\lambda_j(\omega)}\le\frac{(1+\sqrt{1+8L})^2}{16}.
\end{equation}
\end{theorem}

\begin{proof}
\binoppenalty=10000\relpenalty=10000
The area of $\omega$ is $2\pi$. Hersch's inequalities~\cite{Hersch} therefore give $\lambda_1(\omega)\leq4$ and
\begin{equation}\label{eq:hersch-reciprocal}
\sum_{j=1}^{3}\frac1{\lambda_j(\omega)}\geq\frac34.
\end{equation}
For $j=1,2,3$, equation~$(\ref{eq:round})$ gives $\lambda_j(\hat\omega)=4$. Applying equation~$(\ref{eq:interval})$ with $\omega_1=\omega$ and $\omega_2=\hat\omega$, we obtain that
\begin{equation}\label{eq:round-reciprocal-upper}
\frac1{\lambda_j(\omega)}\leq\frac{(1+\sqrt{1+8L})^2}{16}.
\end{equation}
The case $j=1$ of equation~$(\ref{eq:round-reciprocal-upper})$ proves the lower bound in equation~$(\ref{eq:first})$. Summing equation~$(\ref{eq:round-reciprocal-upper})$ over $j=1,2,3$ and using equation~$(\ref{eq:hersch-reciprocal})$ proves equation~$(\ref{eq:sum})$.
\end{proof}

The reciprocal sum also constrains the second positive eigenvalue once the first is controlled. Combining this observation with the classical upper bound gives the following interval.

\begin{theorem}\label{thm:second}
It holds that
\begin{equation}\label{eq:second}
\frac{16}{(1+\sqrt{1+8L})^2}\le\lambda_2(\omega)
\le\frac{16}{\max\{2,\,5-4L-\sqrt{1+8L}\}}.
\end{equation}
\end{theorem}

\begin{proof}
\binoppenalty=10000\relpenalty=10000
The lower bound follows from equation~$(\ref{eq:first})$ and $\lambda_2(\omega)\geq\lambda_1(\omega)$. Since $\lambda_3(\omega)\geq\lambda_2(\omega)$, equations~$(\ref{eq:hersch-reciprocal})$ and~$(\ref{eq:round-reciprocal-upper})$ give
\begin{equation}\label{eq:second-reciprocal-lower}
\frac2{\lambda_2(\omega)}\geq\frac34-\frac1{\lambda_1(\omega)}\geq\frac{5-4L-\sqrt{1+8L}}8.
\end{equation}
The classical sphere bound
\begin{equation}\label{eq:second-area-upper}
\lambda_2(\omega)\int_{\mathbb S^2}\omega\leq16\pi
\end{equation}
gives $\lambda_2(\omega)\leq8$; see~\cite[Theorem 1.2]{KNPP}. Combining equations~$(\ref{eq:second-reciprocal-lower})$ and~$(\ref{eq:second-area-upper})$ gives equation~$(\ref{eq:second})$.
\end{proof}

We next apply the comparison simultaneously to all positive eigenvalues and express the resulting bounds through the counting function.

We recall the definition of the counting function
$$
N_\omega(t):=\#\{j\ge1:\lambda_j(\omega)\le t\}
$$
which counts positive Laplace–Beltrami eigenvalues of $\omega$ with multiplicity. The counting function for the round sphere is explicitly known as
$$
N(t):=k(t)(k(t)+2),
$$
where $k(t)=\lfloor(\sqrt{1+2t}-1)/2\rfloor$.

\begin{theorem}\label{thm:count}
For every $t>0$, it holds that
\begin{equation}\label{eq:countlower}
N_\omega(t)\ge N(4(1+\sqrt{1+2tL})^{-2}t).
\end{equation}
For $0<t<2/L$, it also holds that
\begin{equation}\label{eq:countupper}
N_\omega(t)\le
N((1-tL/2)^{-2}t).
\end{equation}
\end{theorem}

\begin{proof}
\binoppenalty=10000\relpenalty=10000
Put $c(t):=4t/(1+\sqrt{1+2tL})^2$. By equation~$(\ref{eq:interval})$, applied in both orders, we have that
\begin{equation}\label{eq:count-comparison}
\lambda_j(\omega)\geq c(\lambda_j(\hat\omega)),\qquad \lambda_j(\hat\omega)\geq c(\lambda_j(\omega)).
\end{equation}
The function $c$ is strictly increasing, and its inverse on $(0,2/L)$ is $c^{-1}(t)=t/(1-tL/2)^2$. If $\lambda_j(\hat\omega)\leq c(t)$, then the second inequality in equation~$(\ref{eq:count-comparison})$ gives $\lambda_j(\omega)\leq t$. Counting these indices proves equation~$(\ref{eq:countlower})$. If $\lambda_j(\omega)\leq t<2/L$, then the first inequality in equation~$(\ref{eq:count-comparison})$ gives $\lambda_j(\hat\omega)\leq c^{-1}(t)$. Counting these indices proves equation~$(\ref{eq:countupper})$.
\end{proof}

\subsection{Energy refinements of Hersch's inequalities}\label{subsec:energy}

We write $G_\omega$ for the Green operator of $-\Delta_\omega$ on the space of square integrable functions of zero mean integral on the K\"ahler manifold $(\mathbb{P}^1,\omega)$. The eigenvalues of $G_\omega$, listed in decreasing order, are 
$$\lambda_1(\omega)^{-1}\geq\lambda_2(\omega)^{-1}\geq\cdots.$$

\begin{proposition}\label{prop:matrix}
Let $f_1,f_2,f_3$ be linearly independent smooth functions with zero mean integral on the K\"ahler manifold $(\mathbb{P}^1,\omega)$.
Then, it holds that
\begin{equation}\label{eq:matrix-sum}
\sum_{j=1}^{3}\frac1{\lambda_j(\omega)}\geq\operatorname{tr}(A^{-1}B),
\end{equation}
where the matrices $A,B$ are given by
$$
A_{ij}:=\int_{\mathbb S^2}f_if_j\omega,
\qquad
B_{ij}:=Q(G_\omega f_i,G_\omega f_j),
$$
so that $A$ is positive definite.
Moreover, if $f_1^2+f_2^2+f_3^2=1$, then
\begin{equation}\label{eq:matrix-trace}
\lambda_1(\omega)\leq\frac{1}{\operatorname{tr}(B)}\int_{\mathbb{S}^2}\omega.
\end{equation}
\end{proposition}

\begin{proof}
\binoppenalty=10000\relpenalty=10000
Write $\sqrt{A}$ for the positive definite square root of $A$. For $i=1,2,3$, put
$$
h_i:=\sum_{j=1}^{3}(\sqrt{A}^{-1})_{ij}f_j.
$$
The functions $h_1,h_2,h_3$ are orthonormal for the $L^2$ inner product induced by $\omega$ and have zero mean integral. Integration by parts and the min--max principle for $G_\omega$ on their span give
$$
\operatorname{tr}(A^{-1}B)=\sum_{i=1}^{3}\int_{\mathbb S^2}h_iG_\omega h_i\omega\leq\sum_{j=1}^{3}\frac1{\lambda_j(\omega)}.
$$
This proves equation~$(\ref{eq:matrix-sum})$.

Moreover, the largest eigenvalue of $G_\omega$ is $1/\lambda_1(\omega)$. Thus, for $i=1,2,3$, we have that
$$
B_{ii}=\int_{\mathbb S^2}f_iG_\omega f_i\omega\leq\frac1{\lambda_1(\omega)}\int_{\mathbb S^2}f_i^2\omega.
$$
Summing over $i=1,2,3$ and using $f_1^2+f_2^2+f_3^2=1$, we obtain equation~$(\ref{eq:matrix-trace})$.
\end{proof}

The Dirichlet energy of the potential is a negative Sobolev quantity: equation~$(\ref{eq:ibp})$ and the Cauchy--Schwarz inequality give
$$
Q(\phi)=\sup\left\{
4\left|\textstyle\int_{\mathbb S^2}f(\omega-\hat\omega)\right|^2:f\in C^\infty(\mathbb S^2),Q(f)=1\right\}
$$
The right-hand side is four times the squared homogeneous $H^{-1}$ norm of the area-form difference, while equal area removes the ambiguity of constants. The estimate in~\cite{KNPS} also has an explicit constant, but uses an inhomogeneous $W^{-1,2}$ norm and an eigenvalue normalization. Here Theorem~\ref{thm:sphere-energy} gives $$\frac{1}{\lambda_1(\omega)}\geq\frac14+\frac{Q(\phi)}{8\pi},$$ and Theorem~\ref{thm:reciprocal-energy} gives the coefficient $40\pi$ for the average of three reciprocal eigenvalues. This collective estimate does not follow from the first-eigenvalue bound alone. The finite rank comparison itself is a standard variational argument; the energy remainders are the additional estimates.

We finally use the Dirichlet energy of the potential to refine Hersch's inequalities. The next two results impose a centering condition, which removes the degree-$1$ component of the potential.

\begin{theorem}\label{thm:sphere-energy}
Let $\vec{n}$ be the outward-pointing unit normal vector field of the unit $2$-sphere $\mathbb{S}^2:=\{(x_1,x_2,x_3)\in\mathbb{R}^3:x_1^2+x_2^2+x_3^2=1\}$. Suppose that $\omega$ is centered, in the sense that 
$$\int_{\mathbb{S}^2}\vec{n}\omega=0.$$
 Then, it holds that
$$
\lambda_1(\omega)\leq\frac{8\pi}{2\pi+Q(\phi)}<4.
$$
\end{theorem}

\begin{proof}
\binoppenalty=10000\relpenalty=10000
Write $n_1,n_2,n_3$ for the three component functions of $\vec n$. Orthogonality in this proof is with respect to the $L^2$ inner product induced by $\hat\omega$. After adding a constant to $\phi$, we may assume that
$$
\int_{\mathbb S^2}\phi\hat\omega=0.
$$
Since $\Delta n_i=-4n_i$, the centering condition and integration by parts give
\begin{equation}\label{eq:sphere-orthogonality}
0=\int_{\mathbb S^2}n_i\omega=-2\int_{\mathbb S^2}n_i\phi\hat\omega.
\end{equation}
The functions $n_1,n_2,n_3$ span the first eigenspace of $-\Delta$. Thus, by equation~$(\ref{eq:round})$, there exists a unique smooth function $v$, orthogonal to the constants and this eigenspace, such that
\begin{equation}\label{eq:sphere-v}
(\Delta+4)v=-\phi.
\end{equation}
Indeed, the operator $-\Delta-4$ has eigenvalues at least $8$ on this orthogonal complement. The same spectral decomposition gives
\begin{equation}\label{eq:sphere-positive}
Q(\phi,v)=\int_{\mathbb S^2}(\Delta v)^2\hat\omega-4Q(v)\geq8Q(v)\geq0.
\end{equation}

We next construct functions whose Poisson equations have the centered coordinate functions as their right-hand sides.

For $i=1,2,3$, put
\begin{equation}\label{eq:sphere-test}
e_i:=\frac14n_i+\frac12n_i\Delta v-\langle dn_i,dv\rangle.
\end{equation}
Write $\hat g$ for the Riemannian metric induced by $\hat\omega$. The identities $\operatorname{Hess}(n_i)=-2n_i\hat g$ and $\operatorname{Ric}(\hat g)=2\hat g$ give, by the product rule,
\begin{equation}\label{eq:sphere-product}
\Delta\langle dn_i,dv\rangle=\langle dn_i,d\Delta v\rangle-4n_i\Delta v.
\end{equation}
Differentiating equation~$(\ref{eq:sphere-test})$ and using equations~$(\ref{eq:sphere-v})$ and~$(\ref{eq:sphere-product})$, we obtain that
\begin{equation}\label{eq:sphere-poisson}
-\Delta e_i=n_i-\frac12n_i\Delta(\Delta+4)v=n_i\left(1+\frac12\Delta\phi\right).
\end{equation}
In particular, we have that $(-\Delta e_i)\hat\omega=n_i\omega$. Let $c_i$ be the mean value of $e_i$ with respect to $\omega$. By centering, the Cauchy--Schwarz inequality and the Poincar\'e inequality for $\omega$, we obtain that
\begin{equation}\label{eq:sphere-poincare}
Q(e_i)=\int_{\mathbb S^2}(e_i-c_i)n_i\omega\leq\left(\frac{Q(e_i)}{\lambda_1(\omega)}\int_{\mathbb S^2}n_i^2\omega\right)^{1/2}.
\end{equation}
For each $i=1,2,3$, the function $e_i$ is nonconstant by equation~$(\ref{eq:sphere-poisson})$, so we have that $Q(e_i)>0$. The coordinate functions satisfy
\begin{equation}\label{eq:coordinate-identities}
\sum_{i=1}^{3}n_i^2=1,\qquad \sum_{i=1}^{3}n_i\,dn_i=0.
\end{equation}
Squaring equation~$(\ref{eq:sphere-poincare})$, cancelling $Q(e_i)$ and summing over $i=1,2,3$, we obtain from equation~$(\ref{eq:coordinate-identities})$ that
\begin{equation}\label{eq:sphere-summed-poincare}
\lambda_1(\omega)\sum_{i=1}^{3}Q(e_i)\leq2\pi.
\end{equation}

It remains to estimate this sum of energies from below. By equations~$(\ref{eq:sphere-test})$, $(\ref{eq:sphere-poisson})$ and~$(\ref{eq:coordinate-identities})$, we obtain that
$$
\sum_{i=1}^{3}Q(e_i)=\int_{\mathbb S^2}\left(\frac14+\frac12\Delta v\right)\omega=\frac\pi2+\frac14\int_{\mathbb S^2}\Delta v\Delta\phi\hat\omega.
$$
Substituting equation~$(\ref{eq:sphere-v})$ in the last integral and integrating by parts, we obtain from equation~$(\ref{eq:sphere-positive})$ that
\begin{equation}\label{eq:sphere-energy-lower}
\sum_{i=1}^{3}Q(e_i)=\frac{2\pi+Q(\phi)}4+Q(\phi,v)\geq\frac{2\pi+Q(\phi)}4.
\end{equation}
Combining equations~$(\ref{eq:sphere-summed-poincare})$ and~$(\ref{eq:sphere-energy-lower})$ proves the first assertion. Since $L>0$, the potential $\phi$ is nonconstant, and hence we have that $Q(\phi)>0$. This proves the strict inequality.
\end{proof}

The same centering condition also gives a quantitative improvement of the lower bound for the sum of the first three reciprocal eigenvalues.

\begin{theorem}\label{thm:reciprocal-energy}
Under the same assumption as in Theorem~\ref{thm:sphere-energy} that $\omega$ is centered, it holds that
\begin{equation}\label{eq:reciprocal-centered}
\frac{1}{3}\sum_{j=1}^3\frac1{\lambda_j(\omega)}\geq\frac14+\frac{Q(\phi)}{40\pi}.
\end{equation}
\end{theorem}

\begin{proof}
\binoppenalty=10000\relpenalty=10000
Write $\omega=(1+\rho)\hat\omega$, so that $2\rho=\Delta\phi$. We normalize $\phi$ again to have zero mean on the round sphere, namely
$$
\frac{1}{2\pi}\int_{\mathbb S^2}\phi\hat\omega=0.
$$
By centering and equation~$(\ref{eq:sphere-orthogonality})$, both $\rho$ and $\phi$ are orthogonal to the constants and the first eigenspace of $-\Delta$. Let $\Pi$ be the $L^2$-orthogonal projection onto the space of spherical harmonics of degrees at least $2$ on the round sphere. Denote the degree-$\ell$ component of a smooth function $f\in C^\infty(\mathbb S^2)$ by $f[\ell]$.

We shall first verify that, for every $f\in C^\infty(\mathbb S^2)$ satisfying $\Pi (f)=f$, and each $i=1,2,3$, we have that
\begin{equation}\label{eq:reciprocal-multiplier}
Q(\Pi(n_i(3f+2f[2])))\leq15Q(f).
\end{equation}
By rotation, it suffices to consider $n_3$. Denote by $Y_\ell^m$ the $m$-th spherical harmonic of degree $\ell$. Here, we use the usual real spherical harmonics, orthonormal for the $L^2$ inner product induced by $\hat\omega$. Multiplication by $n_3$ satisfies
\begin{equation}\label{eq:harmonic-recurrence}
n_3Y_\ell^m=\alpha_\ell^mY_{\ell-1}^m+\alpha_{\ell+1}^mY_{\ell+1}^m,\qquad \alpha_\ell^m:=\sqrt{\frac{\ell^2-m^2}{4\ell^2-1}},
\end{equation}
where $-\ell\leq m\leq\ell$, and terms with zero coefficient are omitted. For every smooth function $h$, the product rule and integration by parts give
\begin{equation}\label{eq:coordinate-product-energy}
Q(n_3h)=\int_{\mathbb S^2}n_3^2(|dh|^2+4h^2)\hat\omega\leq Q(h)+4\int_{\mathbb S^2}h^2\hat\omega.
\end{equation}
Let $g$ be the sum of the odd-degree components of $f$. Since the function $g$ contains only degrees at least $3$, equations~$(\ref{eq:round})$ and~$(\ref{eq:coordinate-product-energy})$ give
\begin{equation}\label{eq:odd-multiplier}
Q(\Pi(n_3g))\leq Q(n_3g)\leq\frac76Q(g).
\end{equation}

For the even-degree components, put $h:=f-g-f[2]-f[4]$, so that the function $h$ contains only even degrees at least $6$. For simplicity write $$a:=\sqrt{Q(f[2])},\qquad b:=\sqrt{Q(f[4])},\qquad c:=\sqrt{Q(h)}.$$ Since the coefficients $\alpha_\ell^m$ in equation~$(\ref{eq:harmonic-recurrence})$ are largest when $m=0$, we obtain that
\begin{equation}\label{eq:degree-2&4-energy}
Q(\Pi(n_3f[2]))\leq\frac{18}{35}a^2,\qquad Q(\Pi(n_3f[4]))\leq\frac{409}{770}b^2.
\end{equation}
By equations~$(\ref{eq:round})$ and~$(\ref{eq:coordinate-product-energy})$, we also have that
\begin{equation}\label{eq:higher-even-energy}
Q(\Pi(n_3h))\leq Q(n_3h)\leq\frac{22}{21}c^2.
\end{equation}
Notice that, the smooth functions $n_3f[2]$ and $n_3f[4]$ have only degree $3$ in common after projection by $\Pi$. Moreover, the smooth functions $n_3f[4]$ and $n_3h$ have only degree $5$ in common, and the products $n_3f[2]$ and $n_3h$ are orthogonal. Thus, equation~$(\ref{eq:harmonic-recurrence})$ and the Cauchy--Schwarz inequality give that
\begin{equation}\label{eq:even-cross-two-four}
|Q(\Pi(n_3f[2]),\Pi(n_3f[4]))|\leq\frac{4\sqrt6}{35}ab
\end{equation}
and that
\begin{equation}\label{eq:even-cross-four-six}
|Q(\Pi(n_3f[4]),\Pi(n_3h))|\leq\frac{bc}{\sqrt{14}}.
\end{equation}
Using equations~$(\ref{eq:degree-2&4-energy})$--$(\ref{eq:even-cross-four-six})$ and the bounds $409/770\leq3/5$ and $22/21\leq7/6$, we obtain that
\begin{equation}\label{eq:even-multiplier-bound}
\frac19Q(\Pi(n_3(5f[2]+3f[4]+3h)))\leq\frac{10}{7}a^2+\frac35b^2+\frac76c^2+\frac{8\sqrt6}{21}ab+\frac{\sqrt{14}}7bc.
\end{equation}
By virtue of the elementary inequalities
$$
\frac{8\sqrt6}{21}ab\leq\frac5{21}a^2+\frac{32}{35}b^2,\qquad \frac{\sqrt{14}}7bc\leq\frac17b^2+\frac12c^2,
$$
we conclude from equation~$(\ref{eq:even-multiplier-bound})$ that
\begin{equation}\label{eq:even-multiplier}
Q(\Pi(n_3(5f[2]+3f[4]+3h)))\leq15(a^2+b^2+c^2).
\end{equation}
Since multiplication by $n_3$ exchanges odd and even degrees, the two contributions remain orthogonal. Applying equation~$(\ref{eq:odd-multiplier})$ to $3g$ and adding equation~$(\ref{eq:even-multiplier})$, we obtain equation~$(\ref{eq:reciprocal-multiplier})$.

We now apply Proposition~\ref{prop:matrix} to $f_i:=\beta n_i$, for $i=1,2,3$, where
$$
\beta:=\sqrt{\frac3{8\pi}}.
$$
As in Proposition~\ref{prop:matrix}, we define matrices $A,B$ by
$$
A_{ij}:=\int_{\mathbb S^2}f_if_j\omega,
\qquad
B_{ij}:=Q(G_\omega f_i,G_\omega f_j).
$$
By centering, the functions $f_1,f_2,f_3$ have zero mean integral with respect to $\omega$. These functions are orthonormal with respect to $Q$, and we have that
\begin{equation}\label{eq:coordinate-trace}
\operatorname{tr}(A)=\sum_{i=1}^{3}\int_{\mathbb S^2}f_i^2\omega=\frac34.
\end{equation}
Up to a rotation, we may assume that $A$ is diagonal, with $A_{ii}=\alpha_i>0$.

For $i=1,2,3$, we put
$$
e_i:=-G\circ\Pi(n_i\rho),
$$
where the inverse $G:=\Delta^{-1}$ is taken on functions of zero mean on the round $2$-sphere. We also write $\varphi:=3\phi+2\phi[2]$. By integration by parts, we have that $Q(G_\omega f_i,f_j)=A_{ij}$. For every smooth function $v$ satisfying $\Pi(v)=v$, we also obtain that
$$
Q(G_\omega f_i,v)=\beta\int_{\mathbb S^2}n_i\rho v\hat\omega=\beta Q(e_i,v).
$$
Consequently, the smooth function $G_\omega f_i-\alpha_if_i-\beta e_i$ is constant. By orthogonality, we have that $B_{ii}=\alpha_i^2+\beta^2Q(e_i)$, and hence equation~$(\ref{eq:coordinate-trace})$ gives
\begin{equation}\label{eq:reciprocal-matrix-trace}
\operatorname{tr}(A^{-1}B)=\frac34+\frac3{8\pi}\sum_{i=1}^{3}\frac{Q(e_i)}{\alpha_i}.
\end{equation}
It remains to bound the sum in equation~$(\ref{eq:reciprocal-matrix-trace})$ from below in terms of $Q(\phi)$.

To evaluate the pairings with $\Pi(n_i\varphi)$, write $\rho[2](x)=x^\top Wx$ on the unit sphere, where $W$ is a $3\times3$ real symmetric matrix with zero trace. The decomposition of the cubic polynomial $x_i x^\top Wx$ gives
\begin{equation}\label{eq:degree-one-projection}
(I-\Pi)(n_i\rho)=\frac25(Wx)_i.
\end{equation}
Indeed, the polynomial $$x_i x^\top Wx-\frac{2}{5}(x_1^2+x_2^2+x_3^2)(Wx)_i$$ is harmonic of degree $3$, and the higher-degree components of $\rho$ contribute no terms of degrees $0$ or $1$. Multiplying equation~$(\ref{eq:degree-one-projection})$ by $n_i$ and summing, we obtain that
\begin{equation}\label{eq:coordinate-projection}
\sum_{i=1}^{3}n_i\Pi(n_i\rho)=\rho-\frac25\rho[2].
\end{equation}
The projection onto degree $2$ commutes with $\Delta$. Hence, by $2\rho=\Delta\phi$, orthogonality and equation~$(\ref{eq:coordinate-projection})$, we obtain that
\begin{equation}\label{eq:reciprocal-pairing}
\sum_{i=1}^{3}Q(e_i,\Pi(n_i\varphi))=\int_{\mathbb S^2}\left(\rho-\frac25\rho[2]\right)\varphi\hat\omega=3\int_{\mathbb S^2}\rho\phi\hat\omega=-\frac32Q(\phi).
\end{equation}
By the Cauchy--Schwarz inequality and equations~$(\ref{eq:reciprocal-multiplier})$, $(\ref{eq:coordinate-trace})$ and~$(\ref{eq:reciprocal-pairing})$, we conclude that
\begin{equation}\label{eq:reciprocal-cauchy}
\frac14Q(\phi)^2\leq\frac19\sum_{i=1}^{3}\frac{Q(e_i)}{\alpha_i}\sum_{j=1}^{3}Q(\Pi(n_j\varphi))\alpha_j\leq\frac54Q(\phi)\sum_{i=1}^{3}\frac{Q(e_i)}{\alpha_i}.
\end{equation}
Since $Q(\phi)>0$, cancelling $Q(\phi)$ in equation~$(\ref{eq:reciprocal-cauchy})$ and substituting the resulting bound into equation~$(\ref{eq:reciprocal-matrix-trace})$, we obtain equation~$(\ref{eq:reciprocal-centered})$ by Proposition~\ref{prop:matrix}.
\end{proof}

\section{Concluding Remarks}
We have obtained spectral comparisons from the oscillation, the first derivatives and the Laplacian of a potential relating two K\"ahler forms of equal area. The common Dirichlet form and the min--max principle give bounds for reciprocal eigenvalues, their fractional powers and eigenvalue ratios at every positive index.

The estimates also describe spectral stability under changes of potential. For a fixed reference form, bounded oscillation gives a positive lower bound for each positive eigenvalue. As the oscillation tends to zero, the corresponding eigenvalue ratios tend to one. More generally, uniform convergence of potentials, up to additive constants, gives convergence of each positive eigenvalue whenever the limiting form is smooth and positive. These conclusions follow directly from Theorem~\ref{thm:p} and Corollary~\ref{cor:interval}.

For the sphere, the explicit reference spectrum gives bounds for individual eigenvalues, the reciprocal sum and the counting function. Under the centering condition, Theorems~\ref{thm:sphere-energy} and~\ref{thm:reciprocal-energy} give strict improvements of Hersch's inequalities in terms of the Dirichlet energy of a nonconstant potential. For holomorphically immersed curves, the oscillation of the ambient potential gives a common comparison for all induced metrics.
\section*{Acknowledgements}

The author would like to express his most sincere gratitude to Alexei Viktorovich Penskoi, who taught the author complex manifolds and spectral geometry.



\section*{Statements and Declarations}

\noindent\textbf{Funding.} No funding was received to assist with the preparation of this manuscript.

\noindent\textbf{Competing interests.} The author declares no relevant financial or non-financial interests.

\noindent\textbf{Data availability.} Data sharing is not applicable to this article.

\noindent\textbf{Use of generative AI.} During the preparation of this manuscript, the author used OpenAI’s ChatGPT to help handle minor details and grammatical issues. All suggested changes were reviewed by the author, who takes full responsibility for the mathematical accuracy and the final content of this article.

\end{document}